\documentclass{article}

\usepackage{cite}
\usepackage{amssymb}
\usepackage{amsmath}
\usepackage{amsfonts}
\usepackage{amsthm}
\usepackage{fancyhdr}
\usepackage[explicit]{titlesec}
\usepackage{titletoc}
\usepackage{booktabs}
\usepackage[all]{xy}
\usepackage[colorlinks=true,linkcolor=blue]{hyperref}

\usepackage{cancel}

\usepackage{url}

\usepackage{hyperref}

\newtheorem*{theorem*}{Theorem}
\newtheorem{theorem}{Theorem}
\newtheorem{corollary}{Corollary}
\newtheorem{definition}{Definition}

\newtheorem{lemma}{Lemma}

\everymath{\displaystyle}

\begin{document}

\title{A topological semigroup structure on the space of actions modulo weak equivalence.}

\author{Peter Burton}

\date{\today}

\maketitle

\begin{abstract} We introduce a topology on the space of actions modulo weak equivalence finer than the one previously studied in the literature. We show that the product of actions is a continuous operation with respect to this topology, so that the space of actions modulo weak equivalence becomes a topological semigroup.  \end{abstract}

\section{Introduction.}

Let $\Gamma$ be a countable group and let $(X,\mu)$ be a standard probability space. All partitions considered in this note will be assumed to be measurable. If $a$ is a measure-preserving action of $\Gamma$ on $(X,\mu)$ and $\gamma \in \Gamma$ we write $\gamma^a$ for the element of $\mathrm{Aut}(X,\mu)$ corresponding to $\gamma$ under $a$. Let $\mathrm{A}(\Gamma,X,\mu)$ be the space of measure-preserving actions of $\Gamma$ on $(X,\mu)$. We have the following basic definition, due to Kechris.

\begin{definition} For actions $a,b \in \mathrm{A}(\Gamma,X,\mu)$ we say that $a$ is \textbf{weakly contained} in $b$ if for every partition $(A_i)_{i=1}^n$ of $(X,\mu)$, finite set $F \subseteq \Gamma$ and $\epsilon > 0$ there is a partition $(B_i)_{i=1}^n$ of $(X,\mu)$ such that \[ \left \vert \mu\left( \gamma^a A_i \cap A_j\right) - \mu\left(\gamma^b B_i \cap B_j\right)\right\vert < \epsilon \] for all $i,j \leq n$ and all $\gamma \in F$. We write $a \prec b$ to mean that $a$ is weakly contained in $b$. We say $a$ is \textbf{weakly equivalent} to $b$ and write $a \sim b$ if we have both $a \prec b$ and $b \prec a$. $\sim$ is an equivalence relation and we write $[a]$ for the weak equivalence class of $a$. \end{definition}

For more information on the space of actions and the relation of weak equivalence, we refer the reader to \cite{K}.  Let $\mathrm{A}_\sim(\Gamma,X,\mu) = \mathrm{A}(\Gamma,X,\mu) / \sim$ be the set of weak equivalence classes of actions. Freeness is invariant under weak equivalence, so the set $\mathrm{FR}_\sim(\Gamma,X,\mu)$ of weak equivalence classes of free actions is a subset of $\mathrm{A}_\sim(\Gamma,X,\mu)$.\\
\\
Given $[a],[b] \in \mathrm{A}_\sim(\Gamma,X,\mu)$ with representatives $a$ and $b$ consider the action $a \times b$ on $\left(X^2,\mu^2\right)$. We can choose an isomorphism of $\left(X^2,\mu^2 \right)$ with $(X,\mu)$ and thereby regard $a \times b$ as an action on $(X,\mu)$. The weak equivalence class of the resulting action on $(X,\mu)$ does not depend on our choice of isomorphism, nor on the choice of representatives. So we have a well-defined binary operation $\times$ on $\mathrm{A}_{\sim}(\Gamma,X,\mu)$. This is clearly associative and commutative. In Section \ref{sec2} we introduce a new topology on $\mathrm{A}_\sim(\Gamma,X,\mu)$ which is finer than the one studied in \cite{AbEl}, \cite{PBur1} and \cite{RTD}. We call this the fine topology. The goal of this note is to prove the following result.

\begin{theorem} \label{thm1} $\times$ is continuous with respect to the fine topology, so that in this topology $\left(\mathrm{A}_\sim(\Gamma,X,\mu), \times \right)$ is a commutative topological semigroup. \end{theorem}

In \cite{RTD}, Tucker-Drob shows that for any free action $a$ we have $a \times s_\Gamma \sim a$, where $s_\Gamma$ is the Bernoulli shift on $\left( [0,1]^\Gamma,\lambda^\Gamma \right)$ with $\lambda$ being Lebesgue measure. Thus if we restrict attention to the free actions there is additional algebraic structure.

\begin{corollary} With the fine topology, $\left(\mathrm{FR}_\sim(\Gamma,X,\mu), \times \right)$ is a commutative topological monoid. \end{corollary}

\subsection*{Acknowledgements.}

We would like to thank Alexander Kechris for introducing us to this topic and posing the question of whether the product is continuous.

\section{Definition of the fine topology.} \label{sec2}

Fix an enumeration $\Gamma = (\gamma_s)_{s=1}^\infty$ of $\Gamma$. Given $a \in \mathrm{A}(\Gamma,X,\mu)$, $t,k \in \mathbb{N}$ and a partition $\mathcal{A} = (A_i)_{i=1}^k$ of $X$ into $k$ pieces let $M_{t,k}^\mathcal{A}(a)$ be the point in $[0,1]^{t \times k \times k}$ whose $s,l,m$ coordinate is $\mu\left(\gamma_s^a A_l \cap A_m \right)$. Endow $[0,1]^{t \times k \times k}$ with the metric given by the sum of the distances between coordinates and let $d_H$ be the corresponding Hausdorff metric on the space of compact subsets of $[0,1]^{t \times k \times k}$. Let $C_{t,k}(a)$ be the closure of the set \[\left \{ M^{\mathcal{A}}_{t,k}(a): \mathcal{A} \mbox{ is a partition of }X\mbox{ into }k\mbox{ pieces }\right \}.\] We have $a \sim b$ if and only if $C_{t,k}(a) = C_{t,k}(b)$ for all $t,k$. Define a metric $d_f$ on $\mathrm{A}_{\sim}(\Gamma,X,\mu)$ by \[ d_f\left([a],[b]\right) = \sum_{t=1}^\infty \frac{1}{2^t} \left( \sup_{k} d_H\left(C_{t,k}(a),C_{t,k}(b) \right) \right).\] This is clearly finer than the topology on $\mathrm{A}_\sim(\Gamma,X,\mu)$ discussed in the references. 

\begin{definition} The topology induced by $d_f$ is called the the \textbf{fine topology}. \end{definition}

We have $[a_n] \to [a]$ in the fine topology if and only if for every finite set $F \subseteq \Gamma$ and $\epsilon > 0$ there is $N$ so that when $n \geq N$, for every $k \in \mathbb{N}$ and every partition $(A_l)_{l=1}^k$ of $(X,\mu)$ there is a partition $(B_l)_{l=1}^k$ so that \[ \sum_{l,m=1}^k\left \vert \mu \left( \gamma^{a_n} A_l \cap A_m \right) - \mu \left( \gamma^a B_l \cap B_m \right) \right \vert < \epsilon \] for all $\gamma \in F$ and $l,m \leq k$. 

\section{Proof of the theorem.}

We begin by showing a simple arithmetic lemma.

\begin{lemma}\label{lem1} Suppose $I$ and $J$ are finite sets and $(a_i)_{i \in I}, (b_i)_{i \in I}, (c_j)_{j \in J}, (d_j)_{j \in J}$ are sequences of elements of $[0,1]$ with $\sum_{i \in I} a_i = 1$, $\sum_{j \in J} d_j = 1$, $\sum_{i \in I} |a_i - b_i| < \delta$ and $\sum_{j \in J} |c_j - d_j| < \delta$. Then $\sum_{(i,j) \in I \times J} |a_i c_j - b_i d_j| < 2 \delta$. \end{lemma}

\begin{proof} Fix $i$. We have \begin{align*} \sum_{j \in J} |a_i c_j - b_i d_j| & \leq \sum_{j \in J}( |a_i c_j - a_i d_j| + |d_j a_i + d_j b_i|) \\ & = \sum_{j \in J} (a_i |c_j - d_j| + d_j |a_i - b_i| ) \\ & \leq \delta a_i + |a_i - b_i|. \end{align*} Therefore \[ \sum_{(i,j) \in I \times J} |a_i c_j - b_i d_j| \leq \sum_{i \in I} (a_i \delta + |a_i - b_i|) \leq 2 \delta.\]\end{proof}

We now give the main argument.

\begin{proof}[Proof of Theorem \ref{thm1}] Suppose $[a_n] \to [a]$ and $[b_n] \to [b]$ in the fine topology. Fix $\epsilon > 0$ and $ t \in \mathbb{N}$. Let $N$ be large enough so that when $n \geq N$ we have \begin{equation} \max \left( \sup_k d_H\left(C_{t,k}\left(a_n \right),C_{t,k}(a) \right), \hspace{1pt} \sup_k d_H \left( C_{t,k} \left(b_n \right),C_{t,k}(b) \right) \right) < \frac{\epsilon}{4}. \end{equation}

Fix $n \geq N$. Let $k \in \mathbb{N}$ be arbitrary and consider a partition $\mathcal{A} = (A_l)_{l=1}^k$ of $X^2$ into $k$ pieces. Find partitions $\left(D^1_i\right)_{i=1}^p$ and  $\left(D^2_i\right)_{i=1}^q$ of $X$ such that for each $l \leq k$ there are pairwise disjoint sets $I_l \subseteq p \times q$ such that if we write $D_l = \bigcup_{(i,j) \in I_l} D^1_i \times D^2_j $ then \begin{equation} \mu^2 \left( D_l \triangle A_l \right) < \frac{\epsilon}{4k^2}. \end{equation} Write $(\gamma_s)_{s=1}^t = F$. By $(1)$ we can find a partition $\left(E^1_i\right)_{i=1}^p$ of $X$ such that for all $\gamma \in F$ we have \begin{equation} \sum_{i,j=1}^p \left \vert \mu \left( \gamma^{a} D^1_i \cap D^1_j \right) - \mu \left( \gamma^{a_n} E^1_i \cap E^1_j \right ) \right \vert < \frac{\epsilon}{4} \end{equation} and a partition $\left(E^2_i\right)_{i=1}^q$ of $X$ such that for all $\gamma \in F$ we have \begin{equation} \sum_{i,j=1}^q \left \vert \mu \left( \gamma^b D^2_i \cap D^2_j \right) - \mu \left( \gamma^{b_n} E^2_i \cap E^2_j \right ) \right \vert < \frac{\epsilon}{4}. \end{equation} Define a partition $\mathcal{B} = (B_l)_{l=1}^k$ of $X^2$ by setting $B_l = \bigcup_{(i,j) \in I_l}  E^1_i \times E^2_j$. For $\gamma \in F$ we now have

\begin{align*} &\sum_{l,m = 1}^k \left \vert \mu^2( \gamma^{a \times b} D_l \cap D_m) - \mu^2 ( \gamma^{a_n \times b_n} B_l \cap B_m) \right \vert \\ & = \sum_{l,m = 1}^k \Bigg \vert \mu^2 \left( \gamma^{a \times b} \left(\bigcup_{(i_1,j_1) \in I_l} D^1_{i_1} \times D^2_{j_1} \right) \cap \left(\bigcup_{(i_2,j_2) \in I_m} D^1_{i_2} \times D^2_{j_2} \right) \right) \\ & \hspace{1 in}- \mu^2 \left( \gamma^{a_n \times b_n}  \left(\bigcup_{(i_1,j_1) \in I_l} E^1_{i_1} \times E^2_{j_1} \right) \cap \left(\bigcup_{(i_2,j_2) \in I_m} E^1_{i_2} \times E^2_{j_2} \right) \right) \Bigg \vert \end{align*} \begin{align} & = \sum_{l,m = 1}^k \Bigg \vert \mu^2 \left( \left(\bigcup_{(i_1,j_1) \in I_l} \gamma^a D^1_{i_1} \times \gamma^b D^2_{j_1} \right) \cap \left(\bigcup_{(i_2,j_2) \in I_m} D^1_{i_2} \times D^2_{j_2} \right) \right) \nonumber \\ &\hspace{1 in}- \mu^2 \left( \left(\bigcup_{(i_1,j_1) \in I_l} \gamma^{a_n} E^1_{i_1} \times \gamma^{b_n} E^2_{j_1} \right) \cap \left(\bigcup_{(i_2,j_2) \in I_m} E^1_{i_2} \times E^2_{j_2} \right) \right) \Bigg \vert \nonumber \\ & = \sum_{l,m = 1}^k \Bigg \vert \mu^2 \left(\bigcup_{\substack{(i_1,j_1,i_2,j_2) \\ \in I_l \times I_m}} \left( \gamma^a D^1_{i_1} \times \gamma^b D^2_{j_1} \right) \cap \left( D^1_{i_2} \times D^2_{j_2} \right) \right) \nonumber\\ & \hspace{1 in}- \mu^2 \left(\bigcup_{\substack{(i_1,j_1,i_2,j_2) \\ \in I_l \times I_m}}  \left( \gamma^{a_n} E^1_{i_1} \times E^2_{j_1} \right) \cap \left( \gamma^{b_n} E^1_{i_2} \times E^2_{j_2} \right) \right) \Bigg \vert \nonumber\\ & = \sum_{l,m = 1}^k \Bigg \vert \mu^2 \left(\bigcup_{\substack{(i_1,j_1,i_2,j_2) \\ \in I_l \times I_m}} \left( \gamma^a D^1_{i_1} \cap D^1_{i_2} \right) \times \left( \gamma^b D^2_{j_1} \cap D^2_{j_2} \right) \right) \nonumber \\ & \hspace{1 in}- \mu^2 \left(\bigcup_{\substack{(i_1,j_1,i_2,j_2) \\ \in I_l \times I_m}} \left( \gamma^{a_n} E^1_{i_1} \cap E^1_{i_2} \right) \times \left( \gamma^{b_n} E^2_{j_1} \cap E^2_{j_2} \right) \right) \Bigg \vert \nonumber \\ &\leq \sum_{l,m = 1}^k \sum_{\substack{(i_1,j_1,i_2,j_2) \\ \in I_l \times I_m}} \left \vert \mu\left(\gamma^{a} D^1_{i_1} \cap D^1_{i_2}\right) \mu\left( \gamma^{b} D^2_{j_1} \cap D^2_{j_2} \right) - \mu\left(\gamma^{a_n} E^1_{i_1} \cap E^1_{i_2}\right) \mu\left( \gamma^{b_n} E^2_{j_1} \cap E^2_{j_2} \right) \right \vert \nonumber \\ & \leq \sum_{\substack{(i_1,j_1,i_2,j_2) \\ \in p \times q \times p \times q}} \left \vert \mu\left(\gamma^{a} D^1_{i_1} \cap D^1_{i_2}\right) \mu\left( \gamma^{b} D^2_{j_1} \cap D^2_{j_2} \right) - \mu\left(\gamma^{a_n} E^1_{i_1} \cap E^1_{i_2}\right) \mu\left( \gamma^{b_n} E^2_{j_1} \cap E^2_{j_2} \right) \right \vert \nonumber \\ & = \sum_{\substack{(i_1,i_2,j_1,j_2) \\ \in p^2 \times q^2}} \left \vert \mu\left(\gamma^{a} D^1_{i_1} \cap D^1_{i_2}\right) \mu\left( \gamma^{b} D^2_{j_1} \cap D^2_{j_2} \right) - \mu\left(\gamma^{a_n} E^1_{i_1} \cap E^1_{i_2}\right) \mu\left( \gamma^{b_n} E^2_{j_1} \cap E^2_{j_2} \right) \right \vert. \end{align} Now $(3)$ and $(4)$ let us apply Lemma \ref{lem1} with $I = p^2, J = q^2$ and $\delta = \frac{\epsilon}{4}$ to conclude that $(5) \leq \frac{\epsilon}{2}$. Note that for any three subsets $S_1,S_2,S_3$ of a probability space $(Y,\nu)$ we have \begin{align*} |\nu(S_1 \cap S_3) - \nu(S_2 \cap S_3)| &= |\nu(S_1 \cap S_2 \cap S_3) + \nu((S_1 \setminus S_2) \cap S_3) - \nu(S_1 \cap S_2 \cap S_3) - \nu((S_2 \setminus S_1) \cap S_3)|  \\& \leq \nu(S_1 \triangle S_2),\end{align*} hence for any $l,m \leq k$ and any action $c \in \mathrm{A}\left(\Gamma,X^2,\mu^2 \right)$ we have \begin{align*} & \left \vert \mu^2(\gamma^c A_l \cap A_m) - \mu^2(\gamma^c D_l \cap D_m)\right \vert \\ & \leq \left \vert \mu^2(\gamma^c A_l \cap A_m) - \mu^2(\gamma^c D_l \cap A_m) \right \vert + \left \vert \mu^2\left(\gamma^c D_l \cap A_m\right) - \mu^2\left(\gamma^c D_l \cap D_m\right) \right \vert \\ & \leq \mu^2\left(\gamma^c A_l \triangle \gamma^c D_l \right) + \mu^2 (A_m \triangle D_m) \leq \frac{\epsilon}{2 k^2 }, \end{align*} where the last inequality follows from $(2)$. Hence for all $\gamma \in F$, \begin{align*} & \sum_{l,m = 1}^k \left \vert \mu^2(\gamma^{a \times b} A_l \cap A_m) - \mu^2(\gamma^{a_n \times b_n} B_l \cap B_m) \right \vert \\ & \leq \sum_{l,m = 1}^k \left(\left \vert \mu^2(\gamma^a A_l \cap A_m) - \mu^2(\gamma^a D_l \cap D_m)\right \vert + \left \vert \mu^2(\gamma^{a \times b} D_l \cap D_m) - \mu^2(\gamma^{a_n \times b_n} B_l \cap B_m) \right \vert \right) \\ & \leq \sum_{l,m = 1}^k \left( \frac{\epsilon}{2 k^2 } + \left \vert \mu^2(\gamma^{a \times b} D_l \cap D_m) - \mu^2(\gamma^{a_n \times b_n} B_l \cap B_m) \right \vert \right) \\ & \leq \frac{\epsilon}{2} + (5) \leq \epsilon. \end{align*} 

Therefore $M_{t,k}^\mathcal{A}(a \times b)$ is within $\epsilon$ of $M_{t,k}^\mathcal{B}(a_n \times b_n)$ and we have shown that for all $k$, $C_{t,k}(a \times b)$ is contained in the ball of radius $\epsilon$ around $C_{t,k}(a_n \times b_n)$. A symmetric argument shows that if $n \geq N$ then for all $k$, $C_{t,k}(a_n \times b_n)$ is contained in the ball of radius $\epsilon$ around $C_{t,k}(a \times b)$ and thus the theorem is proved.\end{proof}

\bibliographystyle{plain}
\bibliography{bibliography}

Department of Mathematics\\
California Institute of Technology\\
Pasadena, CA 91125\\
\texttt{pjburton@caltech.edu}

\end{document}